\documentclass[reqno,12pt]{amsart}
\usepackage{geometry}
\usepackage{amsmath,amssymb,mathrsfs,color,stackengine}
\usepackage[colorlinks,
linkcolor=red,
anchorcolor=green,
citecolor=blue,
]{hyperref}
\usepackage[upint]{stix}
\usepackage{subfigure}
\usepackage{float}
\usepackage{times}
\usepackage{tikz}
\usetikzlibrary{intersections}
\usepackage{paralist}
\usepackage{tabularx}
\usepackage{booktabs}
\usepackage[labelfont=bf,format=plain,justification=raggedright,singlelinecheck=false]{caption}

\usepackage{comment}

\makeatletter
\def\setliststart#1{\setcounter{\@listctr}{#1}%
  \addtocounter{\@listctr}{-1}}
\makeatother

\makeatletter
\@addtoreset{figure}{section}
\makeatother

\usepackage{calc}
 \newtheorem{The}{Theorem}[section]
 \newtheorem{Cor}[The]{Corollary}
 \newtheorem{Lem}[The]{Lemma}
 \newtheorem{Pro}[The]{Proposition}
 \theoremstyle{definition}

 \numberwithin{equation}{section}

\newcommand{\R}{\mathbb{R}}

\newcommand{\N}{\mathbb{N}}

\newcommand{\SING}{\mbox{\rm Sing}\,}
\newcommand{\CUT}{\mbox{\rm Cut}\,(u)}

\title[Topological and control theoretic properties of Hamilton-Jacobi equations]{Topological and control theoretic properties of Hamilton-Jacobi equations via Lax-Oleinik commutators}
\author{Piermarco Cannarsa, Wei Cheng and Jiahui Hong}
\address{Dipartimento di Matematica, Universit\`a di Roma ``Tor Vergata'', Via della Ricerca Scientifica 1, 00133 Roma, Italy}
\email{cannarsa@mat.uniroma2.it}
\address{Department of Mathematics, Nanjing University, Nanjing 210093, China}
\email{chengwei@nju.edu.cn}
\address{School of Mathematical Sciences, Shanghai Jiao Tong University, Shanghai 200240, China}
\email{hjh9413@sjtu.edu.cn}
\date{\today}
\begin{document}
\maketitle

\begin{abstract}
In the context of weak KAM theory, we discuss the commutators $\{T^-_t\circ T^+_t\}_{t\geqslant0}$ and $\{T^+_t\circ T^-_t\}_{t\geqslant0}$ of Lax-Oleinik operators. We characterize the relation $T^-_t\circ T^+_t=Id$ for both small time and arbitrary time $t$. We show this relation characterizes controllability for  evolutionary Hamilton-Jacobi equation. Based on our previous work on the cut locus of viscosity solution, we refine our analysis of the cut time function $\tau$ in terms of commutators $T^+_t\circ T^-_t-T^+_t\circ T^-_t$ and clarify the structure of the super/sub-level set of the cut time function $\tau$. 
\end{abstract}

\section{Introduction}

Suppose $M$ is a smooth connected and compact manifold without boundary with $TM$ and $T^*M$  the tangent and cotangent bundle of $M$ respectively. Let $H$ be a Hamiltonian on $M$. The study of the forward Hamilton-Jacobi equation 
\begin{equation}\label{eq:HJs-_intro}
	\begin{cases}
		D_tu(t,x)+H(x,D_xu(t,x))=0,\qquad t\in[0,T], x\in M,\\
		u(0,x)=u_0(x),
	\end{cases}
\end{equation}
and the backward Hamilton-Jacobi equation
\begin{equation}\label{eq:HJs+_intro}
	\begin{cases}
		-D_tu(t,x)-H(x,D_xu(t,x))=0,\qquad t\in[0,T], x\in M,\\
		u(T,x)=u_0(x),
	\end{cases}
\end{equation}
plays an important r\^ole in many fields such as the calculus of variation and optimal control (\cite{Bardi_Capuzzo-Dolcetta1997,Cannarsa_Sinestrari_book}), optimal transport (\cite{Villani_book2009,Ambrosio_Brue_Semola_book2021}), Hamiltonian dynamical systems and PDE (\cite{Lions_book,Fathi_book}). In principle, the solution $u(t,x)$ has a representation formula by Lax-Oleinik  evolution, under very general regularity assumptions on $H$ and $u_0$. More precisely, for any $\phi:M\to\R$, we define the abstract Lax-Oleinik operators as
\begin{align*}
	\begin{split}
		T^+_t\phi(x)=&\,\sup_{y\in M}\{\phi(y)-c_t(x,y)\},\\
		T^-_t\phi(x)=&\,\inf_{y\in M}\{\phi(y)+c_t(y,x)\},
	\end{split}
	\qquad t>0, y\in M,
\end{align*}
where $c_t(x,y):[0,+\infty)\times M\times M$ is the \emph{(dynamical) cost function}. If $c_t(x,y)=A_t(x,y)$ is the action of the Lagrangian $L$ associated with $H$, then the solution of \eqref{eq:HJs-_intro} has the representation $u^-(t,x)=T^-_tu_0(x)$ and the solution of \eqref{eq:HJs+_intro} has the representation $u^+(t,x)=T^+_{T-t}u_0(x)$ and the uniqueness issues of \eqref{eq:HJs-_intro} and \eqref{eq:HJs+_intro} are well established. However, we will only touch the part for the case $c_t(x,y)=A_t(x,y)$, the fundamental solution of \eqref{eq:HJs-_intro} and \eqref{eq:HJs+_intro}, in the current paper. 

In this paper, we always suppose $H:T^*M\to\R$ is a Tonelli Hamiltonian and $L:TM\to\R$ is the associated Tonelli Lagrangian. We take the cost function $c_t(x,y)=A_t(x,y)$ with
\begin{align*}
	A_t(x,y)=\inf_{\gamma\in\Gamma^t_{x,y}}\int^t_0L(\gamma,\dot{\gamma})+c[H]\ ds,
\end{align*}
where $c[H]$ is the Ma\~n\'e's critical value and $\Gamma^t_{x,y}$ is the set of absolutely continuous curves $\gamma:[0,t]\to M$ connecting $x$ to $y$. For convenience, we always suppose $c[H]=0$. 

To understand the relation between the solutions of \eqref{eq:HJs-_intro} and \eqref{eq:HJs+_intro}, we will study the \emph{Lax-Oleinik commutators} $\{T^-_t\circ T^+_t\}_{t\geqslant0}$ and $\{T^+_t\circ T^-_t\}_{t\geqslant0}$ instead of the semigroups $\{T^\pm_t\}_{t\geqslant0}$. A key observation is to characterize a function $\phi:M\to\R$ satisfying the relation
\begin{equation}\label{eq:T^-T^+=I_intro}
	T^-_t\circ T^+_t\phi(x)=\phi(x),\quad t>0,  x\in M.
\end{equation}

Our first result is 

\medskip
\noindent\textbf{Attainable set} (Theorem \ref{pro:equiv_T^-T^+}) : Suppose $\phi\in\text{\rm SCL}\,(M)$, the space of semiconcave functions with linear modulus on $M$, and $t>0$. Then the following statements are equivalent.
\begin{enumerate}[\rm (1)]
	\item $T^-_{t}\circ T^+_{t}\phi=\phi$.
	\item There exists a lower semicontinuous function $\psi:M\to\R$ such that $\phi=T^-_{t}\psi$.
	\item For any $x\in M$ and $p\in D^*\phi(x)$, let $\gamma(s)=\pi_x\Phi_H^{s}(x,p)$, $s\in[-t,0]$\footnote{$D^*\phi(x)$ is the set of reachable gradients of $\phi$ at $x$ and $\Phi_H^{s}$ is the Hamiltonian flow of $H$.}. Then
	\begin{align*}
		T^+_{t}\phi(\gamma(-t))=\phi(x)-\int^0_{-t}L(\gamma,\dot{\gamma})\ ds.
	\end{align*}
\end{enumerate}

It is worth noting that Theorem \ref{pro:equiv_T^-T^+} can be viewed as a controllability result for equation \eqref{eq:HJs-_intro}. More precisely, our result ensures that a function $\phi$ can be \emph{reached} by some initial data $\psi$ for $T^-_t$ if and only if $\phi$ satisfies condition \eqref{eq:T^-T^+=I_intro}. We also provide a detailed description of the relation between this controllability result and underlying dynamics in Section 3.4. This kind of problems have been already addressed in the literature.
 For instance, in \cite{BCJS1999}, for autonomous equations of the form
\begin{equation}\label{eq:HJ_a}
	D_tu+H(D_xu)=0,\quad (t,x)\in(0,T)\times\R^n,
\end{equation}
it shown that any solution $u$ is of class $C^1$ if solving the equation backward from $u(T,\cdot)$ one finds the value $u(0,\cdot)$ at $t=0$ (compare with Theorem \ref{pro:equiv_T^-T^+}). As for controllability, let us recall Proposition 9 in \cite{Ancona_Cannarsa_Nguyen2016_1} describes a subset of $C(M,\R)$ that is attainable for the negative Lax-Oleinik seigroup associated with the autonomous equation \eqref{eq:HJ_a} (see also Proposition 5 in \cite{Ancona_Cannarsa_Nguyen2016_2}) for a local version of this result that applies to more general Hamiltonians). A similar attainability result for \eqref{eq:HJ_a} is obtained in \cite{Esteve-Yague_Zuazua2023} for nonsmooth Hamiltonians. In the case $t\ll1$, \eqref{eq:T^-T^+=I_intro} is closely related to a theorem by Marie-Claude Arnaud (\cite{Arnaud2011}) on the evolution of the 1-graph of a semiconcave function under the Hamiltonian flow, as well as  Lasry-Lions type regularization in the context of weak KAM theory as first studied by Patrick Bernard (\cite{Bernard2007}) (see Proposition \ref{pro:T^-T^+1}).

The following result gives a new characterization of weak KAM solutions of the stationary Hamilton-Jacobi
\begin{equation}\label{eq:HJs_intro}
	H(x,Du(x))=0,\qquad x\in M.
\end{equation}
In other words, weak KAM solutions are exactly those functions that are reachable for $T^-_t$ for all $t>0$.

\medskip
\noindent\textbf{Reversibility} (Theorem \ref{pro:tau2_infty}) : Suppose $\phi\in\text{\rm SCL}\,(M)$. Then $T^-_t\circ T^+_t\phi=\phi$ for all $t\geqslant0$ if and only if $\phi$ is a weak KAM solution of \eqref{eq:HJs_intro}.  

\medskip

We remark that $T^-_{t}\circ T^+_{t}\phi=\phi$ implies that $\phi$ is semiconcave automatically. If the equality $T^-_{t}\circ T^+_{t}\phi=\phi$ holds, then there exists  $\psi:M\to\R$ which is semiconvex such that
\begin{align*}
	\phi=T^-_{t}\psi,\quad \psi=T^+_t\phi.
\end{align*}
Such a pair $(\phi,\psi)$ is also called an \emph{admissible Kantorovich pair} (for the cost function $c_t(x,y)=A_t(x,y)$) in the theory of optimal transport (see \cite{Bernard_Buffoni2007a,Bernard_Buffoni2007b,Villani_book2009}).

The second part of this paper is on the structure of the cut locus of weak KAM solutions to \eqref{eq:HJs_intro}. 
Let $u$ be a weak KAM solution of \eqref{eq:HJs_intro}. For any $x\in M$, the cut time function of $u$ is defined by
\begin{align*}
	\tau(x):=\sup\{t\geqslant0: \exists\gamma\in C^1([0,t],M), \gamma(0)=x, u(\gamma(t))-u(x)=A_t(x,\gamma(t))\}
\end{align*}
Moreover, $\tau(x)$ can be related to the commutators of the Lax-Oleinik semigroups as follows (see Section 3.3)
\begin{align*}
	\tau(x)=\sup\{t\geqslant0: (T^-_t\circ T^+_t-T^+_t\circ T^-_t)u(x)=0\}.
\end{align*}
Then, we define the cut locus of $u$, $\CUT$, and the Aubry set of $u$, $\mathcal{I}\,(u)$, as follows
\begin{align*}
	\CUT=\{x:\tau(x)=0\},\qquad \mathcal{I}\,(u)=\{x:\tau(x)=+\infty\}.
\end{align*}
In \cite{Cannarsa_Cheng_Fathi2017,Cannarsa_Cheng_Fathi2021} we proved that the complement of $\mathcal{I}\,(u)$ is homotopically equivalent to $\CUT$.

Given a weak KAM solution $u$ of \eqref{eq:HJs_intro}, we define
\begin{align*}
	G^*(u):=&\,\{(x,p): x\in M, p\in D^*u(x)\subset T^*_xM\},\\
	G^{\#}(u):=&\,\{(x,p): x\in M, p\in D^+u(x)\setminus D^*u(x)\subset T^*_xM\}.
\end{align*}
The following result characterizes the super/sub-level sets of the cut time function $\tau$ for $t>0$.

\medskip

\noindent\textbf{Super/sub-level sets of $\tau$} (Theorem \ref{thm:level bi-lip})
\hfill
\begin{enumerate}[\rm (1)]
	\item For any $t>0$, the set $\{x\in M: \tau(x)\geqslant t\}$ is bi-Lipschitz homeomorphic to $G^{*}(u)$.
	\item There exists $t_0>0$ such that for all $0<t<t_0$, the set $\{x\in M: \tau(x)<t\}$ is bi-Lipschitz homeomorphic to $G^{\#}(u)$.
\end{enumerate}

\medskip

Finally, we remark that these results have essential applications to our recent work on the intrinsic construction of generalized characteristics and strict singular characteristics \cite{Cannarsa_Cheng3}. These results clarify the relations among certain features of Hamilton-Jacobi equations such as irreversibility, non-commutativity, and singularity.

The paper is organized as follows. In section 2, we give a brief introduction to weak KAM theory and Hamilton-Jacobi equations. In Section 3, we discuss equality $T^-_{t}\circ T^+_{t}\phi=\phi$ for both small time and arbitrary time $t$, from both functional level and underlying dynamics. We also discuss the long time behavior of the operators $T^-_{t}\circ T^+_{t}$. We finally analyze the commutator of  $T^+_{t}$ and $T^-_{t}$ and its implications in the structure of the cut locus. 

\medskip

\noindent\textbf{Acknowledgements.} Piermarco Cannarsa was supported in part by the National Group for Mathematical Analysis, Probability and Applications (GNAMPA) of the Italian Istituto Nazionale di Alta Matematica ``Francesco Severi'' and by the Excellence Department Project awarded to the Department of Mathematics, University of Rome Tor Vergata, CUP E83C23000330006. Wei Cheng is partly supported by National Natural Science Foundation of China (Grant No. 12231010). The authors also appreciate Kai Zhao for helpful discussion. 

\section{Preliminaries}

\begin{table}[h]
    \caption{Notation}
    \begin{tabularx}{\textwidth}{p{0.22\textwidth}X}
    \toprule
    $M$ & compact and connected smooth manifold without boundary\\
    $C(M)$ & the space of continuous functions on $M$\\
    $\text{SCL}\,(M)$ & the class of semiconcave functions with linear modulus on $M$\\
    $C^{1,1}(M)$ & the space of $C^1$ functions on $M$ with Lipschitz continuous differentials\\
    $c[H]$ & Ma\~n\'e's critical value with respect to $H$\\
    $D^{\pm}\phi(x)$ & the superdifferential and subdifferential of a function $\phi$\\
    $D^*\phi(x)$ & the set of reachable gradients of a function $\phi$\\
    $\Phi^t_H$ & the Hamiltonian flow associated to the Hamiltonian $H$\\
    $\overline{\gamma_1\gamma_2}$ & a curve which is the juxtaposition of two  curves $\gamma_1$ and $\gamma_2$ \\
    $\mathcal{I}\,(u)$ & the Aubry set of an individual weak KAM solution $u$\\
    $\text{Cut}\,(u)$ & the cut locus of an individual weak KAM solution $u$\\
    $\text{Sing}\,(\phi)$ & the set of the points of non-differentiability of a function $\phi$\\
    $\tau(x)$ & the cut time function of a given weak KAM solution\\
    \bottomrule
    \end{tabularx}
\end{table}

For any smooth connected and closed manifold $M$, we call $H=H(x,p):T^*M\to\R$ a \emph{Tonelli Hamiltonian} if $H$ is of class $C^2$, the function $H(x,\cdot)$ is strictly convex and uniformly superlinear. The associated Lagrangian is defined by Legendre transformation:
\begin{align*}
	L(x,v)=\sup_{p\in T^*_xM}\{\langle p,v\rangle_x-H(x,p)\},\quad x\in M, v\in T_xM.
\end{align*}
As a real-valued function on $TM$, $L$ is called a Tonelli Lagrangian. That is $L$ is also of class $C^2$ and the function $L(x,\cdot)$ is strictly convex and uniformly superlinear. From classical calculus of variation, we say that a curve $\gamma:[a,b]\to M$ is an extremal if it satisfies the Euler-Lagrange equations locally. We denote by $\{\Phi^t_H\}_{t\in\R}$ the flow associated to the Hamiltonian vector field of $H$.

Given a Tonelli Hamiltonian $H$ with its associated Tonelli Lagrangian $L$. \emph{Ma\~n\'e's critical value of $H$} is the unique real number $c[H]$ such that the Hamilton-Jacobi equation
\begin{equation}\label{eq:HJ_c}
	H(x,Du(x))=c[H],\quad x\in M
\end{equation}
admits a viscosity solution. We always suppose $c[H]=0$ in this paper for convenience. For   $\phi\in C(M)$, we define the Lax-Oleinik operators as
\begin{align*}
	\begin{split}
		T^+_t\phi(x)=&\,\sup_{y\in M}\{\phi(y)-A_t(x,y)\},\\
		T^-_t\phi(x)=&\,\inf_{y\in M}\{\phi(y)+A_t(y,x)\},
	\end{split}
	\qquad t>0, x\in M,
\end{align*}
A function $\phi:M\to\R$ is called a negative (resp. positive) \emph{weak KAM solution} if $T^-_t\phi=\phi$ (resp. $T^+_t\phi=\phi$) for all $t\geqslant0$. It is known that weak KAM solutions and viscosity solutions coincide. We usually omit the adjective ``negative'' if we do not mention the positive one in the context. A pair of functions $(u^-,u^+)$ on $M$ is called a weak KAM pair if $u^-$ and $u^+$ are negative and positive weak KAM solutions, respectively, and they coincide on the projected Mather set (see more in \cite[Section 5.1]{Fathi_book}, \cite{Fathi1997_2}). 

Any viscosity sub-solution of \eqref{eq:HJ_c} is called a \emph{critical sub-solution}. The following proposition characterizes critical sub-solutions.

\begin{Pro}[\cite{Fathi_book}]\label{pro:subsolution}
Given $\phi\in C(M,\R)$, the following properties are equivalent:
\begin{enumerate}[\rm (a)]
	\item $\phi$ is a sub-solution of \eqref{eq:HJ_c}.
	\item The inequality $\phi(y)-\phi(x)\leqslant A_t(x,y)$ holds for each $t>0$ and each $(x,y)\in M\times M$.
	\item The function $[0,\infty)\ni t\mapsto T^-_t\phi(x)$ is non-decreasing for each $x\in M$.
	\item The function $[0,\infty)\ni t\mapsto T^+_t\phi(x)$ is non-increasing for each $x\in M$.
\end{enumerate}
\end{Pro}

Let $\Omega\subset\R^n$ be a bounded and open set, a function $\phi:\Omega\to\R$ is a semiconcave function (with linear modulus) if there is a constant $C>0$ such that 
\begin{equation}\label{eq:SCC}
\lambda \phi(x)+(1-\lambda)\phi(y)-\phi(\lambda x+(1-\lambda)y)\leqslant\frac C2\lambda(1-\lambda)|x-y|^2
\end{equation}
for any $x,y\in\Omega$ and $\lambda\in[0,1]$.  Any constant $C$ that satisfies the above inequality  is called a \emph{semiconcavity constant} for $u$ in $\Omega$.

Let $\phi:\Omega\subset\R^n\to\R$ be a continuous function. We recall that, for any $x\in\Omega$, the closed convex sets
\begin{align*}
D^-\phi(x)&=\left\{p\in\R^n:\liminf_{y\to x}\frac{\phi(y)-\phi(x)-\langle p,y-x\rangle}{|y-x|}\geqslant 0\right\},\\
D^+\phi(x)&=\left\{p\in\R^n:\limsup_{y\to x}\frac{\phi(y)-\phi(x)-\langle p,y-x\rangle}{|y-x|}\leqslant 0\right\}.
\end{align*}
are called the {\em subdifferential} and {\em superdifferential} of $\phi$ at $x$, respectively.

Let now $\phi:\Omega\to\R$ be locally Lipschitz. We recall that a vector $p\in\R^n$ is said to be a {\em reachable} (or {\em limiting}) {\em gradient} of $\phi$ at $x$ if there exists a sequence $\{x_n\}\subset\Omega\setminus\{x\}$, converging to $x$, such that $\phi$ is differentiable at $x_k$ for each $k\in\N$ and $\lim_{k\to\infty}D\phi(x_k)=p$. The set of all reachable gradients of $\phi$ at $x$ is denoted by $D^{\ast}\phi(x)$.

\begin{Pro}[\cite{Cannarsa_Sinestrari_book}]
\label{criterion-Du_semiconcave2}
Let $\phi:\Omega\to\R$ be a continuous function. If there exists a constant $C>0$ such that, for any $x\in\Omega$, there exists $p\in\R^n$ such that
\begin{equation}\label{criterion_for_lin_semiconcave}
\phi(y)\leqslant \phi(x)+\langle p,y-x\rangle+\frac C2|y-x|^2,\quad \forall y\in\Omega,
\end{equation}
then $\phi$ is semiconcave with constant $C$ and $p\in D^+\phi(x)$.
Conversely, if $\phi$ is semiconcave  in $\Omega$ with constant $C$, then \eqref{criterion_for_lin_semiconcave} holds for any $x,y\in\Omega$ such that $[x,y]\subseteq\Omega$ and $p\in D^+\phi(x)$.
\end{Pro}

Proposition \ref{criterion-Du_semiconcave2} is useful to understand semiconcave functions and their superdifferential on manifolds. Suppose $M$ is endowed with a Riemannian metric $g$, with $d$ the associated Riemannian distance. We call $\phi:M\to\R$ semiconcave with constant $C$ if 
\begin{align*}
	t\phi(x)+(1-t)\phi(y)-\phi(\gamma(t))\leqslant t(1-t)C d^2(x,y),\forall x,y\in M, \gamma\in\Gamma_{x,y},
\end{align*}
where $\Gamma_{x,y}$ is the set of all geodesics $\gamma:[0,1]\to M$ connecting $x$ to $y$. For any $x\in M$, we say that $p_x\in D^+\phi(x)$ if $p_x$ is a one form such that
\begin{align*}
	\phi(y)\leqslant\phi(x)+p_x(\dot{\gamma}(0))+\frac{C}2d^2(x,y),\qquad\forall\gamma\in\Gamma_{x,y}.
\end{align*}
One can show that this definition of semiconcavity is independent of the choice of the Riemannian metric, even though semiconcavity constants may change. The reader can refer to \cite{Villani_book2009,Udricste_Udricste1994,Fathi_Figalli2010} for more on semiconcave and convex functions on manifolds.

The following Moreau-Yosida-Lasry-Lions regularization type result in the context of weak KAM theory is due to Patrick Bernard. 

\begin{Pro}[\cite{Bernard2007}]\label{pro:LL}
Let $\phi\in\text{\rm SCL}\,(M)$. Then there exists $t_0>0$ such that $T^+_{t_0}\phi\in C^{1,1}(M)$.
\end{Pro}

\begin{Lem}\label{lem1}
Suppose $f,g:\R^n\to\R$, $f\geqslant g$. If $x_0\in\R^n$ is such that $f(x_0)=g(x_0)$, then $D^+f(x_0)\subset D^+g(x_0)$, $D^-g(x_0)\subset D^-f(x_0)$.
\end{Lem}

\begin{proof}
If $p\in D^+f(x_0)$, then for any $\theta\in\R^n$,
\begin{align*}
	\limsup_{t\to 0^+}\frac{g(x_0+t\theta)-g(x_0)}{t}\leqslant
		\limsup_{t\to 0^+}\frac{f(x_0+t\theta)-f(x_0)}{t}\leqslant \langle p,\theta\rangle,
\end{align*}
which implies that $p\in D^+g(x_0)$. Thus, $D^+f(x_0)\subset D^+g(x_0)$,  $D^-g(x_0)\subset D^-f(x_0)$ by a similar reasoning.
\end{proof}

\begin{Pro}\label{pro:cave geq vex}
	Suppose $u:\R^n\to\R$ is semiconcave and $v:\R^n\to\R$ is semiconvex with the same constant $C$. Assume $u\geqslant v$, and $A=\{x\in\R^n:u(x)=v(x)\}$. Then
	\begin{enumerate}[\rm (1)]
		\item $Du(x)=Dv(x)$, $\forall x\in A$.
		\item $Du=Dv$ are $4C$-Lipschitz on $A$.
	\end{enumerate}
\end{Pro}

\begin{proof}
For any $x\in A$, Lemma \ref{lem1} implies $D^+u(x)\subset D^+v(x)$. Since $u$ is  semiconcave with linear modulus, we have that $D^+u(x)\neq\varnothing$. Thus, $Dv(x)$ exists and $D^+u(x)\subset D^+v(x)=\{Dv(x)\}$. So, $Du(x)=Dv(x)$. This completes the proof of (1).

We now turn to the proof of (2). For any $x_1,x_2\in A$, $\theta\in\R^n$, we have
\begin{align*}
	&\,\langle Du(x_1)-Du(x_2),\theta\rangle \leqslant \langle Dv(x_1)-Du(x_2),\theta \rangle\\
	\leqslant&\,v(x_1+\theta)-v(x_1)+\frac{C}{2}|\theta|^2-u(x_2+\theta)+u(x_2)+\frac{C}{2}|\theta|^2\\
	\leqslant&\,\big[-\frac{1}{2}v(x_1)-\frac{1}{2}v(x_1+2\theta)+v(x_1+\theta)\big]+\big[-\frac{1}{2}v(x_1)-\frac{1}{2}v(2x_2-x_1)+v(x_2)\big]\\
	&\,+\big[\frac{1}{2}u(x_1+2\theta)+\frac{1}{2}u(2x_2-x_1)-u(x_2+\theta)\big]+C|\theta|^2\\
	\leqslant&\,\frac{C}{2}|\theta|^2+\frac{C}{2}|x_2-x_1|^2+\frac{C}{2}|x_2-x_1-\theta|^2+C|\theta|^2\\
	\leqslant&\,C|x_2-x_1|^2+C|x_2-x_1|\cdot|\theta|+2C|\theta|^2.
\end{align*}
It follows that
\begin{align*}
	&\,|Du(x_1)-Du(x_2)|=\frac{1}{|x_2-x_1|}\max_{|\theta|=|x_2-x_1|}\langle Du(x_1)-Du(x_2),\theta \rangle\\
	\leqslant&\,\frac{1}{|x_2-x_1|}(C|x_2-x_1|^2+C|x_2-x_1|\cdot|x_2-x_1|+2C|x_2-x_1|^2)\\
	=&\,4C|x_2-x_1|.
\end{align*}
Therefore, $Du=Dv$ is $4C$-Lipschitz on $A$.
\end{proof}

\section{Lax-Oleinik commutators, controllability and singularity}
\label{sec:commutativity}

In this section we consider the Hamilton-Jacobi equation
\begin{equation}\label{eq:HJs}\tag{HJ$_s$}
	H(x,Du(x))=0,\qquad x\in M.
\end{equation} 

\subsection{Commutators of Lax-Oleinik semigroup}

As in Proposition \ref{pro:subsolution}, the Lax-Oleinik semigroups $\{T^{\pm}_t\phi\}$ have monotonicity properties in $t$ only when $\phi$ is a continuous critical sub-solution of \eqref{eq:HJs}. In this section, we will consider the operators $T^-_t\circ T^+_t$ and $T^+_t\circ T^-_t$ instead. We begin by analyzing the property
\begin{align*}
	T^-_t\circ T^+_t\phi(x)=\phi(x),\quad t>0, x\in M,
\end{align*}
for $\phi\in\text{\rm SCL}\,(M)$. The study of the operators $T^-_t\circ T^+_t$ and $T^+_t\circ T^-_t$ has already been introduced in the context of weak KAM theory in \cite{Bernard2010,Bernard2012}. 
However, for application to controllability and singularity issues, we need to complete the results by Bernard with some new observations.

\begin{Lem}\label{lem:1}
\hfill
\begin{enumerate}[\rm (1)]
	\item If $\phi:M\to\R$ is upper semicontinuous, then $T^-_t\circ T^+_t\phi\geqslant\phi$ for all $t>0$, and $T^-_t\circ T^+_t\phi$ is non-decreasing with respect to $t$. Moreover, for any $t>0$ and $x\in M$, 
	\begin{align*}
		T^-_t\circ T^+_t\phi(x)=\phi(x)
	\end{align*}
	if and only if there exists $\gamma\in C^1([0,t],M)$ with $\gamma(t)=x$ such that
	\begin{equation}\label{eq:T^-T^+1}
		T^+_t\phi(\gamma(0))=\phi(x)-\int^t_0L(\gamma,\dot{\gamma})\ ds.
	\end{equation}
	\item If $\phi:M\to\R$ is lower semicontinuous, then $T^+_t\circ T^-_t\phi\leqslant\phi$ for all $t>0$, and $T^+_t\circ T^-_t\phi$ is non-increasing with respect to $t$. Moreover, for any $t>0$ and $x\in M$, 
	\begin{align*}
		T^+_t\circ T^-_t\phi(x)=\phi(x)
	\end{align*}
	if and only if there exists $\gamma\in C^1([0,t],M)$ with $\gamma(0)=x$ such that
	\begin{equation}\label{eq:T^+T^-1}
		T^-_t\phi(\gamma(t))=\phi(x)+\int^t_0L(\gamma,\dot{\gamma})\ ds.
	\end{equation}
    \item If $\phi:M\to\R$ is upper semicontinuous and $t>0$, then
    \begin{align*}
    	T^+_t\circ T^-_t\circ T^+_t\phi=T^+_t\phi.
    \end{align*}
    If $\phi:M\to\R$ is lower semicontinuous and $t>0$, then
    \begin{align*}
    	T^-_t\circ T^+_t\circ T^-_t\phi=T^-_t\phi.
    \end{align*}
\end{enumerate}
\end{Lem}

\begin{proof}
(1) Let $\phi:M\to\R$ be upper semicontinuous. For any $x\in M$ and $t\geqslant0$ we have that
\begin{align*}
	T^-_t\circ T^+_t\phi(x)=&\,\inf_{y\in M}\{T^+_t\phi(y)+A_t(y,x)\}\\
	=&\,\inf_{y\in M}\{\sup_{z\in M}\{\phi(z)-A_t(y,z)\}+A_t(y,x)\}\\
	\geqslant&\,\phi(x)
\end{align*}
by taking $z=x$. Thus, if $t_2\geqslant t_1>0$, then
\begin{align*}
	T^-_{t_2}\circ T^+_{t_2}\phi=T^-_{t_1}\circ T^-_{t_2-t_1}\circ T^+_{t_2-t_1} \circ T^+_{t_1}\phi\geqslant T^-_{t_1}\circ T^+_{t_1}\phi.
\end{align*}
Now, let $t>0$ and let $x\in M$ be such that $T^-_t\circ T^+_t\phi(x)=\phi(x)$. Then there exists a $C^1$ curve $\gamma:[0,t]\to M$, $\gamma(t)=x$ such that 
\begin{align*}
	\phi(x)=T^-_t\circ T^+_t\phi(x)= T^+_t\phi(\gamma(0))+\int^t_0L(\gamma,\dot{\gamma})\ ds.
\end{align*}
Thus, \eqref{eq:T^-T^+1} holds. Conversely, if there exists a $C^1$ curve $\gamma:[0,t]\to M$, with $\gamma(t)=x$, which satisfies \eqref{eq:T^-T^+1}, then
\begin{align*}
	\phi(x)=T^+_t\phi(\gamma(0))+\int^t_0L(\gamma,\dot{\gamma})\ ds\geqslant T^-_t\circ T^+_t\phi(x).
\end{align*}
The inequality above is an equality since $T^-_t\circ T^+_t\phi\geqslant\phi$. This completes the proof of (1). The proof of (2) is similar and will be omitted.

 (3) Suppose $\phi:M\to\R$ is upper semicontinuous and $t>0$. By (1), we have that
\begin{align*}
	T^+_t\circ T^-_t\circ T^+_t\phi=T^+_t\circ(T^-_t\circ T^+_t)\phi\geqslant T^+_t\phi.
\end{align*}
On the other hand, property (2) implies that
\begin{align*}
	T^+_t\circ T^-_t\circ T^+_t\phi=(T^+_t\circ T^-_t)\circ T^+_t\phi\leqslant T_t^+\phi.
\end{align*}
Thus, there holds $T^+_t\circ T^-_t\circ T^+_t\phi=T^+_t\phi$. Similarly, if $\phi:M\to\R$ is lower semicontinuous and $t>0$, then $T^-_t\circ T^+_t\circ T^-_t\phi=T^-_t\phi$.
\end{proof}

Before proceeding further, we recall a result from \cite{Arnaud2011} which describes the evolution of the 1-graph 
\begin{align*}
	\text{\rm graph}\,(D^+\phi)=\{(x,p): x\in M, p\in D^+\phi(x)\subset T^*M\}
\end{align*}
under the Hamiltonian flow $\{\Phi_H^t\}$ for short time  (see (3) of the following proposition). 

\begin{Pro}\label{pro:T^-T^+1}
Let $\phi\in\text{\rm SCL}\,(M)$ and let $t_0>0$ be such that $T^+_{t_0}\phi\in C^{1,1}(M)$. Then the following holds true.
\begin{enumerate}[\rm (1)]
	\item $T^+_t\phi\in C^{1,1}(M)$ for all $t\in(0,t_0]$ and $T^-_{t}\circ T^+_{t}\phi=\phi$ for all $t\in[0,t_0]$.
	\item $T^+_t\phi=T^-_{t_0-t}\circ T^+_{t_0}\phi$ for all $t\in[0,t_0]$.
	\item {\rm (Arnaud)} We have that
	\begin{equation}\label{eq:graph_evo}
		\text{\rm graph}\,(DT^+_t\phi)=\Phi_H^{-t}(\text{\rm graph}\,(D^+\phi)),\qquad\forall t\in(0,t_0].
	\end{equation}
	\item Let $u(t,x)=T^+_t\phi(x)$ for $(t,x)\in[0,t_0]\times M$. Then $u$ is of class $C^{1,1}_{\rm loc}$ on $(0,t_0)\times M$  and it is a viscosity solution of the Hamilton-Jacobi equation
	\begin{equation}\label{eq:HJ+}\tag{HJ$_e^+$}
		\begin{cases}
			D_tu-H(x,D_xu)=0,\qquad(t,x)\in(0,t_0)\times M;\\
			u(0,x)=\phi(x),\qquad x\in M.
		\end{cases}
	\end{equation}
\end{enumerate} 
\end{Pro}

\begin{proof}
Suppose $\phi\in\text{\rm SCL}\,(M)$, $t_0>0$ and $T^+_{t_0}\phi\in C^{1,1}(M)$. The fact that $T^+_{t}\phi\in C^{1,1}(M)$ for all $t\in(0,t_0]$ is known (\cite{Bernard2007}). Fix $x\in M$ and $t\in(0,t_0]$. There exists $y\in M$ such that
\begin{align*}
	T^-_{t}\circ T^+_{t}\phi(x)=T^+_{t}\phi(y)+A_{t}(y,x).
\end{align*}
Being semiconcave and attaining its minimum at $y$, $T_{t}^{+}\phi(\cdot)+A_{t}(\cdot,x)$ must be differentiable at $y$. Since $A_{t}(\cdot,x)$ is itself differentiable by Proposition \ref{pro:regularity}, $T_{t}^{+}\phi$ must be differentiable at $y$ as well. By Fermat's rule we have 
\begin{align*}
	L_v(\xi(0),\dot{\xi}(0))=-D_yA_{t}(y,x)=DT^+_{t}\phi(y),
\end{align*}
where $\xi\in\Gamma^{t}_{y,x}$ is the unique minimal curve for $A_{t}(y,x)$. On the other hand, there exists a unique $z\in M$ such that $T^+_{t}\phi(y)=\phi(z)-A_{t}(y,z)$ since $T^+_{t}\phi\in C^{1,1}(M)$. Then, 
\begin{align*}
	DT^+_{t}\phi(y)=L_v(\eta(0),\dot{\eta}(0))
\end{align*}
where $\eta\in\Gamma^{t}_{y,z}$ is the unique minimal curve for $A_{t}(y,z)$. Thus, $\xi(0)=\eta(0)=y$ and $L_v(\eta(0),\dot{\eta}(0))=L_v(\xi(0),\dot{\xi}(0))$. Since $\xi$ and $\eta$ satisfy the Euler-Lagrange equation with the same initial conditions, we conclude that $\xi\equiv\eta$, and so that $z=x$. Therefore,
\begin{align*}
	T^-_{t}\circ T^+_{t}\phi(x)=\phi(z)-A_{t}(y,z)+A_{t}(y,x)=\phi(x).
\end{align*}
and (1) follows.

By (1) we conclude that 
\begin{equation}\label{eq:T^-T^+2}
	T^-_{t_0-t}\circ T^+_{t_0}\phi=T^-_{t_0-t}\circ T^+_{t_0-t}\circ T^+_{t}\phi=T^+_{t}\phi,\qquad t\in(0,t_0].
\end{equation}
The proof is completed noting that Statement (3) is known (see \cite{Arnaud2011}) and (4) is obvious. 
\end{proof}

For any $\phi\in\text{\rm SCL}\,(M)$, define two functions $\tau_1,\tau_2:\text{\rm SCL}\,(M)\to[0,\infty]$
\begin{equation}\label{eq:def of tau}
    \tau_1(\phi):=\,\sup\{t>0:  T^+_t\phi\in C^{1,1}(M)\},\qquad \tau_2(\phi):=\,\sup\{t\geqslant0: T^-_t\circ T^+_t\phi=\phi\}.
\end{equation}
Let $\pi_x:T^*M\to M$ be the canonical projection onto $M$. 

\begin{Cor}\label{cor:tau12}
Suppose $\phi\in\text{\rm SCL}\,(M)$.
\begin{enumerate}[\rm (1)]
	\item $\tau_2(\phi)\geqslant\tau_1(\phi)>0$.
	\item We have
	\begin{align*}
		\{t>0:  T^+_t\phi\in C^{1,1}(M)\}=&\,(0,\tau_1(\phi)),\\
		\{t\geqslant0: T^-_t\circ T^+_t\phi=\phi\}=&\,
		\begin{cases}[0,\tau_2(\phi)],&\tau_2(\phi)<\infty;\\
		    [0,\infty),&\tau_2(\phi)=\infty.
		\end{cases}
	\end{align*}
\end{enumerate}	
\end{Cor}

\begin{proof}
(1) is a consequence of Proposition \ref{pro:T^-T^+1}, and (2) is easily follows from the definition of $\tau_1(\phi)$ and $\tau_2(\phi)$. 
\end{proof}

\begin{Cor}\label{cor:graph}
Suppose $\phi\in\text{\rm SCL}\,(M)$, $\tau_1(\phi)<\infty$. Then $T^+_{\tau_1(\phi)}\phi$ is semiconvex and the following properties hold.
\begin{enumerate}[\rm (1)]
	\item $T^+_t\phi=T^-_{\tau_1(\phi)-t}\circ T^+_{\tau_1(\phi)}\phi$ for all $t\in[0,\tau_1(\phi)]$, and $T^+_t\phi\in C^{1,1}(M)$ for all $t\in(0,\tau_1(\phi))$.
	\item $\text{\rm graph}\,(DT^+_t\phi)=\Phi_H^{-t}(\text{\rm graph}\,(D^+\phi))=\Phi_H^{\tau_1(\phi)-t}(\text{\rm graph}\,(D^- T^+_{\tau_1(\phi)}\phi))$ for all $t\in(0,\tau_1(\phi))$.
	\item Let $u(t,x)=T^+_t\phi(x)$ for $(t,x)\in[0,\tau_1(\phi)]\times M$. Then $u$ is of class $C^{1,1}_{\rm loc}$ on $(0,\tau_1(\phi))\times M$  and it is a viscosity solution of the Hamilton-Jacobi equation
	\begin{align*}
		\begin{cases}
			D_tu-H(x,D_xu)=0,\qquad(t,x)\in(0,\tau_1(\phi))\times M;\\
			u(0,x)=\phi(x),\qquad x\in M.
		\end{cases}
	\end{align*}
\end{enumerate}
\end{Cor}

\begin{proof}
The proof follows from Proposition \ref{pro:T^-T^+1} and the definition of $\tau_1(\phi)$ and $\tau_2(\phi)$.
\end{proof}

The following proposition characterizes the condition $T^-_t\circ T^+_t\phi=\phi$ for general $t>0$. 

\begin{The}\label{pro:equiv_T^-T^+}
Suppose $\phi\in\text{\rm SCL}\,(M)$ and $t_0>0$. Then the following statements are equivalent.
\begin{enumerate}[\rm (1)]
	\item $T^-_{t_0}\circ T^+_{t_0}\phi=\phi$.
	\item There exists a lower semicontinuous function $\psi:M\to\R$ such that $\phi=T^-_{t_0}\psi$.
	\item For any $x\in M$ and $p\in D^*\phi(x)$, we have that 
	\begin{align*}
		T^+_{t_0}\phi(\gamma(-t_0))=\phi(x)-\int^0_{-t_0}L(\gamma,\dot{\gamma})\ ds.
	\end{align*}
	where $\gamma(s)=\pi_x\Phi_H^{s}(x,p)$, $s\in[-t_0,0]$.
\end{enumerate}
\end{The}


\begin{proof}
Taking $\psi=T^+_{t_0}\phi$, it follows that (1) implies (2). Now, suppose there exists a lower semicontinuous function $\psi:M\to\R$ such that $\phi=T^-_{t_0}\psi$. By Lemma \ref{lem:1} (3) we obtain
\begin{align*}
	T^-_{t_0}\circ T^+_{t_0}\phi=T^-_{t_0}\circ T^+_{t_0}\circ T^-_{t_0}\psi=T^-_{t_0}\psi=\phi.
\end{align*}
The proof of the  equivalence of (1) and (2) is thus complete.

Since $(3)\Rightarrow (1)$ is obvious in view of Lemma \ref{lem:1}, it is sufficient to prove that $(1)\Rightarrow (3)$. We suppose that $x\in M$, $p\in D^*\phi(x)$ and let $\gamma(s)=\pi_x\Phi_H^{s}(x,p)$, $s\in[-t_0,0]$. Since $T^-_{t_0}\circ T^+_{t_0}\phi=\phi$, we have that $p\in D^*T^-_{t_0}\circ T^+_{t_0}\phi(x)$. Thus, $\gamma$ is an extremal such that
\begin{align*}
	\phi(x)=T^-_{t_0}\circ T^+_{t_0}\phi(x)=T^+_{t_0}\phi(\gamma(-t_0))+\int^0_{-t_0}L(\gamma,\dot{\gamma})\ ds
\end{align*}
thanks to \cite[Theorem 6.4.9]{Cannarsa_Sinestrari_book}. This leads to our conclusion.
\end{proof}

\subsection{Long time behavior of the operators $T^-_t\circ T^+_t$.}

Recall that $\tau_1$ and $\tau_2$ are defined in \eqref{eq:def of tau}.
\begin{The}\label{pro:tau2_infty}
Suppose $\phi\in\text{\rm SCL}\,(M)$. 
\begin{enumerate}[\rm (1)]
	\item $\tau_2(\phi)=+\infty$ if and only if $\phi$ is a weak KAM solution of \eqref{eq:HJs}.
	\item $\tau_1(\phi)=+\infty$ if and only if $\phi\in C^{1,1}(M)$ and satisfies \eqref{eq:HJs}.
\end{enumerate} 
\end{The}

\begin{proof}
Suppose $\phi$ is a weak KAM solution of \eqref{eq:HJs}. Then $T^-_t\phi=\phi$ for all $t\geqslant0$. By Proposition \ref{pro:equiv_T^-T^+} we have that $T^-_{t}\circ T^+_{t}\phi=\phi$ for all $t\geqslant0$. It follows that $\tau_2(\phi)=+\infty$. 

Conversely, suppose $\tau_2(\phi)=+\infty$.  For any $x\in M$, $p\in D^*\phi(x)$ let $\gamma(s)=\pi_x\Phi_H^{s}(x,p)$ with $s\in(-\infty,0]$. By Theorem \ref{pro:equiv_T^-T^+} we conclude that
\begin{align*}
	T^+_{t}\phi(\gamma(-t))=\phi(x)-\int^0_{-t}L(\gamma,\dot{\gamma})\ ds,\quad\forall t\geqslant0.
\end{align*}
Thus, $A_t(\gamma(-t),x)=\int^0_{-t}L(\gamma,\dot{\gamma})\ ds$ for all $t\geqslant0$. By Lemma \ref{lem:app_appr_2} we have that $H(x,p)=0$. This implies that $\phi$ satisfies \eqref{eq:HJs} almost everywhere. Since $H$ is convex with respect to the $p$-variable, then $\phi$ is a viscosity solution of \eqref{eq:HJs} (see \cite[Proposition 5.3.1]{Cannarsa_Sinestrari_book}). This completes the proof of (1).

Now, we turn to the proof of (2). If $\phi\in C^{1,1}(M)$ satisfies \eqref{eq:HJs}, then $T^+_t\phi=\phi\in C^{1,1}(M)$ for all $t>0$, which implies $\tau_1(\phi)=+\infty$. If $\tau_1(\phi)=+\infty$, then for any $x\in M$, $p\in D^+\phi(x)$ let $\gamma(s)=\pi_x\Phi_H^s(x,p)$, $s\in(-\infty,0]$. Theorem \ref{pro:equiv_T^-T^+} leads to
\begin{align*}
	T^+_{t}\phi(\gamma(-t))=\phi(x)-\int^0_{-t}L(\gamma,\dot{\gamma})\ ds,\quad\forall t\geqslant0.
\end{align*}
Thus $A_t(\gamma(-t),x)=\int^0_{-t}L(\gamma,\dot{\gamma})\ ds$ for all $t\geqslant0$. By Lemma \ref{lem:app_appr_2} we have that $H(x,p)=0$ for all $p\in D^{+}\phi(x)$. Since $H(x,\cdot)$ is strictly convex, we deduce that $D^+\phi(x)$ is a singleton, that is, $p=D\phi(x)$. Therefore, $\phi$ is a $C^{1}$ solution of \eqref{eq:HJs}. It is well known that any $C^{1}$ solution $\phi$ of \eqref{eq:HJs} is of class $C^{1,1}$. This leads to our conclusion.
\end{proof}

%

Still assuming $c[H]=0$ we recall that the Peierls barrier in classical weak KAM theory is defined by
\begin{equation}\label{eq:P B}
	h(x,y)=\liminf_{t\to+\infty}A_t(x,y),\qquad x,y\in M.
\end{equation}
For $\phi\in C(M)$, let
\begin{align*}
	S^-\phi=\lim_{t\to+\infty}T^{-}_{t}\phi,\quad
	S^+\phi=\lim_{t\to+\infty}T^{+}_{t}\phi.
\end{align*}
The liminf in \eqref{eq:P B} is indeed a limit and the limits in the definition of $S^{\pm}$ exist by a well known result of Fathi (see, \cite{Fathi_book,Fathi1998_1}).

\begin{Pro}\label{pro:lim_exchange}
Let $\phi\in C(M)$.
\begin{enumerate}[\rm (1)]
	\item We have the following relations:
	\begin{align*}
		S^-\phi(x)=\inf_{y\in M}\{\phi(y)+h(y,x)\},\qquad S^+\phi(x)=\sup_{y\in M}\{\phi(y)-h(x,y)\},\qquad x\in M.
	\end{align*}
	\item We also have
	\begin{align*}
		S^-S^+\phi=\lim_{t\to+\infty}T^{-}_{t}T^{+}_{t}\phi,\qquad S^+S^-\phi=\lim_{t\to+\infty}T^{+}_{t}T^{-}_{t}\phi.
	\end{align*}
	\item $S^-S^+\phi\geqslant\phi$. For any $x\in M$, $S^-S^+\phi(x)=\phi(x)$ if and only if there exists $y\in M$ such that
	\begin{align*}
		S^+\phi(y)=\phi(x)-h(y,x).
	\end{align*}
	\item $S^+S^-\phi\leqslant\phi$. For any $x\in M$, $S^+S^-\phi(x)=\phi(x)$ if and only if there exists $y\in M$ such that
	\begin{align*}
		S^-\phi(y)=\phi(x)+h(x,y).
	\end{align*}
\end{enumerate}
\end{Pro}

\begin{proof}
We know from Fleming's lemma \cite[Theorem 4.4.3]{Fathi_book} that $A_t(\cdot,\cdot)$ is equi-Lipschitz for $t\geqslant 1$, and the convergence in \eqref{eq:P B} is uniform. Thus, for any $x\in M$, we have
\begin{align*}
S^-\phi(x)&=\lim_{t\to+\infty}T^{-}_{t}\phi(x)=\lim_{t\to+\infty}\inf_{y\in M}\{\phi(y)+A_t(y,x)\}\\
&=\inf_{y\in M}\lim_{t\to+\infty}\{\phi(y)+A_t(y,x)\}=\inf_{y\in M}\{\phi(y)+h(y,x)\}.
\end{align*}
The proof of the second equality in (1) is similar and this completes the proof of (1).

Notice that by the non-expansive property of Lax-Oleinik semigroups (\cite{Fathi_book}), for any $t>0$ there holds
\begin{align*}
	\|T^-_{t}S^+\phi-T^-_{t}T^+_{t}\phi\|_{\infty}&=\|T^-_{t}S^+\phi-T^-_{t}(T^+_{t}\phi)\|_{\infty}\\
	&\leqslant\|S^+\phi-(T^+_{t}\phi)\|_{\infty}\to 0,\quad \text{as}\ t\to +\infty.
\end{align*}
Therefore, we have $S^-S^+\phi=\lim_{t\to+\infty}T^-_{t}S^+\phi=\lim_{t\to+\infty}T^{-}_{t}T^{+}_{t}\phi$. The proof of the second equality of (2) is similar.

The proof of (3) and (4) is similar to that of Lemma \ref{lem:1} and will be omitted.
\end{proof}

\subsection{Cut locus}

Recall Mather's barrier function $B:M\to\R$ introduced in \cite{Mather1993},
\begin{align*}
	B(x):=u^-(x)-u^+(x),\qquad x\in M,
\end{align*}
where $(u^-,u^+)$ is a weak KAM pair. We introduce an analogue of the barrier function $B$. For any weak KAM solution $u$ of \eqref{eq:HJs}, let $B:[0,+\infty)\times M\to\R$,
\begin{align*}
	B(t,x):=u(x)-T^+_tu(x),\qquad (t,x)\in [0,+\infty)\times M.
\end{align*}
Due to Theorem \ref{pro:tau2_infty} we have
\begin{align*}
	B(t,x)=(T^-_t\circ T^+_t-T^+_t\circ T^-_t)u^-(x).
\end{align*}
We denote by $\CUT$ the \emph{cut locus} of $u$. That is, $x\in\CUT$ if any calibrated curve ending at $x$ can not be extended further as a calibrated curve. More precisely, let $\tau:M\to\R$ be the \emph{cut time function} of $u$, for any $x\in M$,
\begin{align*}
	\tau(x):=&\,\sup\{t\geqslant0: \exists\gamma\in C^1([0,t],M), \gamma(0)=x, u(\gamma(t))-u(x)=A_t(x,\gamma(t))\}.
\end{align*}
Then $\CUT=\{x\in M: \tau(x)=0\}$. Recall that
\begin{align*}
	\mathcal{I}(u)=\{x\in M:\text{there exists a } u \text{-calibrated curve } \gamma:[0,+\infty)\to M \text{ such that } \gamma(0)=x\}
\end{align*}
The following properties of {\color{red}the} cut locus and cut time function are essentially known.

\begin{Pro}[\cite{Cannarsa_Cheng_Fathi2017,Cannarsa_Cheng_Fathi2021}]\label{pro_cut_time}
Suppose $u$ is a weak KAM solution of \eqref{eq:HJs}.
\begin{enumerate}[\rm (1)]
	\item Given $t>0$ and $x\in M$, then $T^+_tu(x)=u(x)$ if and only if there exists a $u$-calibrated curve $\gamma:[0,t]\to M$ such that $\gamma(0)=x$. 
	\item $\tau(x)=\sup\{t\geqslant0: B(t,x)=0\}$ for all $x\in M$. 
	\item $\tau$ is upper-semicontinuous and $\CUT$ is a $G_\delta$-set.
	\item $\CUT=\{x:\tau(x)=0\}$ and $\mathcal{I}(u)=\{x:\tau(x)=+\infty\}$.
\end{enumerate}	
\end{Pro}

Given a weak KAM pair $(u^-,u^+)$ of \eqref{eq:HJs}, we define
\begin{align*}
	u(t,x)=T^+_{-t}u^-(x),\qquad (t,x)\in(-\infty,0]\times M.
\end{align*}
It is clear that $u^+(x)\leqslant u(t,x)\leqslant u^-(x)$ for all $(t,x)\in(-\infty,0]\times M$. Define
\begin{align*}
	G^*(u^-):=&\,\{(x,p): x\in M, p\in D^*u^-(x)\subset T^*_xM\},\\
	G^{\#}(u^-):=&\,\{(x,p): x\in M, p\in D^+u^-(x)\setminus D^*u^-(x)\subset T^*_xM\},
\end{align*}
then $\text{graph}\,(D^+u^-)=G^*(u^-)\cup G^{\#}(u^-)$. Let 
\begin{align*}
	\mathscr{F}^*(u^-)=&\,\{\gamma:(-\infty,0]\to M: \gamma(s)=\pi_x\Phi_H^s(x,p), (x,p)\in G^*(u^-)\}\\
	\mathcal{A}^*(u^-)=&\,\{(t,x)\in(-\infty,0]\times M: \exists\gamma\in\mathscr{F}^*(u^-), \gamma(t)=x\}\\
	\mathcal{A}^{\#}(u^-)=&\,((-\infty,0]\times M)\setminus\mathcal{A}^*(u^-).
\end{align*}

\begin{Pro}\label{pro:level point}
\hfill
\begin{enumerate}[\rm (1)]
	\item For any $(t,x)\in\mathcal{A}^*(u^-)\setminus(\{0\}\times M)$ there exists a unique $\gamma\in\mathscr{F}^*(u^-)$ with $\gamma(t)=x$. Let $p=L_v(\gamma,\dot{\gamma})$, then we have
	\begin{align*}
		u(t,x)=u^-(x)=u^-(\gamma(0))-\int^0_tL(\gamma,\dot{\gamma})\ ds,\\
		D_xu(t,x)=Du^-(x)=p(t),\\
		D_tu(t,x)=-H(x,D_xu(t,x))=0.
	\end{align*}
    \item For any $(t,x)\in\mathcal{A}^{\#}(u^-)$, let $\gamma:[t,0]\to M$, $\gamma(t)=x$ be a maximizer for $u(t,x)$, that is,
    \begin{align*}
    	u(t,x)=u^-(\gamma(0))-\int_{t}^{0}L(\gamma,\dot{\gamma})\ ds,
    \end{align*}
    and let $p=L_v(\gamma,\dot{\gamma})$. Then $(\gamma(0),p(0))\in G^{\#}(u^-)$, $\gamma(0)\in\SING(u^-)$, $H(x,p(t))<0$, and
    \begin{align*}
    	u^+(x)<u(t,x)<u^-(x).
    \end{align*}
\end{enumerate}
\end{Pro}

\begin{proof}
The uniqueness of $\gamma$ follows from the fact that calibrated curves cannot cross in the interior of the time interval. Now we have
\begin{align*}
	u^-(x)=u^-(\gamma(0))-\int_{t}^{0}L(\gamma,\dot{\gamma})\ ds\leqslant u(t,x)\leqslant u^-(x),
\end{align*}
which implies $u(t,x)=u^-(x)$. Notice that $u$ and $u^-$ are differentiable on the interior of the calibrated curve. Thus, $D_xu(t,x)=Du^-(x)=p(t)$. Finally, $(\gamma(0),p(0))\in G^{*}(u^-)$ implies
\begin{align*}
	D_tu(t,x)=-H(x,D_xu(t,x))=-H(\gamma(t),p(t))=-H(\gamma(0),p(0))=0.
\end{align*}

Now, we turn to prove (2). First, by calculus of variations we know that $p(0)\in D^{+}u^-(\gamma(0))$. The definition of $\mathcal{A}^{\#}(u^-)$ implies that $p(0)\notin D^{*}u^-(\gamma(0))$. It follows that $(\gamma(0),p(0))\in G^{\#}(u^-)$. Therefore, we have that $\gamma(0)\in\SING(u^-)$, and $H(x,p(t))=H(\gamma(0),p(0))<0$. It is well known that $u^+(x)\leqslant u(t,x)\leqslant u^-(x)$. If $u(t,x)=u^-(x)$, then we can choose $\gamma:[t,0]\to M$, $\gamma(t)=x$ to be a calibrated curve for $u(t,x)$. Thus,
\begin{align*}
	u^-(x)=u(t,x)=u^-(\gamma(0))-\int_{t}^{0}L(\gamma,\dot{\gamma})\ ds.
\end{align*}
So, $\gamma$ is in fact a calibrated curve for $u^-$, which implies $L_v(\gamma(0),\dot{\gamma}(0))\in D^*u^-(\gamma(0))$. As a result, $(t,x)\in\mathcal{A}^{*}(u^-)$. This leads to a contradiction. So we have $u(t,x)<u^-(x)$. If $u^+(x)=u(t,x)$, we know that $(0,D^{-}u^+(x))\in D^-u(t,x)=\mbox{\rm co}\,D^{*}u(t,x)$. Since $H(x,p(t))<0$ for all calibrated curves of $u(t,x)$, it is easy to see that
\begin{align*}
	q=-H(x,p)>0,\qquad \forall (q,p)\in D^{*}u(t,x).
\end{align*}
This leads to a contradiction with $(0,D^{-}u^+(x))\in\mbox{\rm co}\,D^{*}u(t,x)$. Therefore, $u^+(x)<u(t,x)$.
\end{proof}

\begin{Pro}\label{pro:level set}
\hfill
\begin{enumerate}[\rm (1)]
	\item For any $t\geqslant0$ and $x\in M$,
	\begin{align*}
		(-t,x)\in\mathcal{A}^*(u^-)\quad\Longleftrightarrow\quad T^+_tu^-(x)=u^-(x)\quad\Longleftrightarrow\quad \tau(x)\geqslant t.
	\end{align*}
	That is
	\begin{align*}
		\mathcal{A}^*(u^-)=&\,\{(t,x)\in(-\infty,0]\times M: \tau(x)\geqslant -t\}\\
		=&\,\{(t,x)\in(-\infty,0]\times M: u(t,x)=u^-(x)\}.
	\end{align*}
	Moreover, $Du^-=Du$ is locally Lipschitz on $\mathcal{A}^*(u^-)\setminus(\{0\}\times M)$.
	\item For any $t\geqslant0$ and $x\in M$,
	\begin{align*}
		(-t,x)\in\mathcal{A}^{\#}(u^-)\quad\Longleftrightarrow\quad T^+_tu^-(x)<u^-(x)\quad\Longleftrightarrow\quad \tau(x)<t.
	\end{align*}
	That is
	\begin{align*}
		\mathcal{A}^{\#}(u^-)=&\,\{(t,x)\in(-\infty,0]\times M: \tau(x)<-t\}\\
		=&\,\{(t,x)\in(-\infty,0]\times M: u(t,x)<u^-(x)\}.
	\end{align*}
\end{enumerate}
\end{Pro}

\begin{proof}
It is sufficient to only prove (1) since (2) is an immediate consequence of (1). The first equivalence follows from Proposition \ref{pro:level point} and the second one  from Proposition \ref{pro_cut_time} (2). These equivalences imply that
\begin{align*}
	\mathcal{A}^*(u^-)=&\,\{(t,x)\in(-\infty,0]\times M: \tau(x)\geqslant -t\}\\
	=&\,\{(t,x)\in(-\infty,0]\times M: u(t,x)=u^-(x)\}.
\end{align*}
By Proposition \ref{pro:cave geq vex}, we know that $Du^-=Du$ is locally Lipschitz on $\mathcal{A}^*(u^-)\setminus(\{0\}\times M)$.
\end{proof}

\begin{The}\label{thm:level bi-lip}
\hfill
\begin{enumerate}[\rm (1)]
	\item For any $t>0$, there is a bi-Lipschitz homeomorphism between $\{x\in M: \tau(x)\geqslant t\}$ and $G^{*}(u^-)$.
	\item For any $t\in(0,\tau_1(u^-)]$, there is a bi-Lipschitz homeomorphism between  $\{x\in M: \tau(x)<t\}$ and $G^{\#}(u^-)$.
\end{enumerate}
\end{The}

\begin{proof}
For any $t>0$, Proposition \ref{pro:level set} (1) implies there is a $C^1$ diffeomorphism between $\{(x,Du^-(x)):\tau(x)\geqslant t\}$ and $\Phi_H^{-t}(G^{*}(u^-))$, and the projection
\begin{align*}
	\pi_x:\{(x,Du^-(x)):\tau(x)\geqslant t\}\to\{x:\tau(x)\geqslant t\}
\end{align*}
is a bi-Lipschitz homeomorphism. Therefore, $\{x\in M: \tau(x)\geqslant t\}$ is bi-Lipschitz homeomorphic to $G^{*}(u^-)$.

Now we turn to the proof of (2). For any $t\in(0,\tau_1(u^-)]$, Proposition \ref{pro:level set} (2) and Proposition \ref{pro:T^-T^+1} imply that $\Phi_H^{-t}$ determines a $C^1$ diffeomorphism between $\{(x,Du^-(x)):\tau(x)<t\}$ and $\Phi_H^{-t}(G^{\#}(u^-))$, and the projection
\begin{align*}
	\pi_x:\{(x,Du^-(x)):\tau(x)<t\}\to\{x:\tau(x)<t\}
\end{align*}
is a bi-Lipschitz homeomorphism. Therefore, $\{x\in M: \tau(x)<t\}$ is bi-Lipschitz homeomorphic to $G^{\#}(u^-)$.
\end{proof}

\begin{Cor}\label{cor:finite t source}
Suppose $\psi:M\to\R$ is lower semicontinuous, $t_0>0$. Let
\begin{align*}
	u_{\psi}(t,x)=T^{-}_{t_0+t}\psi(x),\qquad (t,x)\in[-t_0,0]\times M.
\end{align*}
Then the following statements are equivalent
\begin{enumerate}[\rm (1)]
	\item $u^-=T^{-}_{t_0}\psi$.
	\item $u_{\psi}:[-t_0,0]\times M\to\R$ is a viscosity solution to the Hamilton-Jacobi equation
	\begin{align*}
		\begin{cases}
			D_tu+H(x,D_xu)=0,&(t,x)\in [-t_0,0)\times M,\\
			u(0,x)=u^-(x),& x\in M.
		\end{cases}
	\end{align*}
    \item $\psi(x)=u^-(x)$ for all $x$ satisfying $\tau(x)\geqslant t_0$, and $T^{+}_{t_0}u^-(x)\leqslant\psi(x)$ otherwise.
\end{enumerate}
\end{Cor}

\begin{proof}
The conclusion follows directly from Proposition \ref{pro:tau2_infty}, Proposition \ref{pro:level set}, and Corollary \ref{cor:cave t source} below.
\end{proof}

%
%

\subsection{More on controllability and underlying dynamics}

In fact, the equivalence between (1) and (2) in Theorem \ref{pro:equiv_T^-T^+} implies that $T^-_{t}\circ T^+_{t}\phi=\phi$ is a necessary and sufficient condition for $\phi$ to be \emph{reachable} in time $t$ from some initial data $\psi$ assigned to the evolutionary Hamilton-Jacobi equation
\begin{align*}
	\begin{cases}
		D_tu(t,x)+H(x,D_xu(t,x))=0,\\
		u(0,x)=\psi(x).
	\end{cases}
\end{align*}

To understand more on this controllability problem and the underlying dynamics in the context of weak KAM theory, we introduce the following sets, directly related to Theorem  \ref{pro:equiv_T^-T^+}. Recall $\tau_2(\phi)>0$ and suppose $0<t\leqslant\tau_2(\phi)$. We define 
\begin{align*}
	\mathscr{F}(\phi,t):=&\,\bigg\{\gamma\in C^1([-t,0],M): T^-_t\circ T^+_t\phi(\gamma(0))=T^+_t\phi(\gamma(-t))+\int^0_{-t}L(\gamma,\dot{\gamma})\ ds\bigg\},\\
	\mathcal{A}(\phi,t):=&\,\{(s,\gamma(s))\in[-t,0]\times M: \gamma\in\mathscr{F}(\phi,t)\},\\
	\mathscr{F}^*(\phi,t):=&\,\bigg\{\gamma\in C^1([-t,0],M): \gamma(s)=\pi_x\Phi^s_H(x,p), s\in[-t,0], x\in M, p\in D^*\phi(x)\bigg\},\\
	\mathcal{A}^*(\phi,t):=&\,\{(s,\gamma(s))\in[-t,0]\times M: \gamma\in\mathscr{F}^*(\phi,t)\}.
\end{align*}
Given $t_0\in(0,\tau_2(\phi)]$ and $\phi\in\text{\rm SCL}\,(M)$, we set
\begin{equation}\label{eq:u_breve_u}
	\breve{u}(t,x)=T^+_{-t}\phi(x),\qquad u(t,x)=T^-_{t_0+t}\circ T^+_{t_0}\phi(x),\qquad (t,x)\in[-t_0,0]\times M.
\end{equation}

\begin{Pro}\label{pro:control1}
Suppose $\phi\in\text{\rm SCL}\,(M)$, $t_0\in(0,\tau_2(\phi)]$ and $u,\breve{u}$ are defined in \eqref{eq:u_breve_u}.
\begin{enumerate}[\rm (1)]
	\item For any $(t,x)\in\mathcal{A}(\phi,t_0)\setminus(\{-t_0,0\}\times M)$, there exists a unique $\gamma\in\mathscr{F}(\phi,t_0)$ such that $\gamma(t)=x$. Moreover, if $p=L_v(\gamma,\dot{\gamma})$ is the dual arc of $\gamma$, then
	\begin{align*}
		u(t,x)=&\,\breve{u}(t,x)=\phi(\gamma(0))-\int^0_tL(\gamma,\dot{\gamma})\ ds,\\
		D_xu(t,x)=&\,D_x\breve{u}(t,x)=p(t).
	\end{align*}
	\item $\breve{u}\leqslant u$, and
	\begin{align*}
		\mathcal{A}(\phi,t_0)=\{(t,x)\in[-t_0,0]\times M: u(t,x)=\breve{u}(t,x)\}.
	\end{align*}
	$Du=D\breve{u}\in\text{\rm Lip}_{loc}\,(\mathcal{A}(\phi,t_0)\setminus(\{-t_0,0\}\times M))$. In particular, $\mathcal{A}(\phi,t_0)=[-t_0,0]\times M$ if $t_0\in(0,\tau_1(\phi)]$.
	\item $\mathscr{F}^*(\phi,t_0)\subset\mathscr{F}(\phi,t_0)$ and $\mathcal{A}^*(\phi,t_0)\subset\mathcal{A}(\phi,t_0)$.
\end{enumerate}
\end{Pro}

\begin{proof}
(1) The uniqueness of $\gamma$ follows from the fact that minimizers cannot cross. Notice that $\gamma$ is a maximizer for $T^+_{-t}\phi(x)$ and a minimizer for $T^-_{t_0+t}T^+_{t_0}\phi(x)$. Thus, we have
\begin{align*}
	\breve{u}(t,x)&=T^+_{-t}\phi(x)=\phi(\gamma(0))-\int_{t}^{0}L(\gamma,\dot{\gamma})\ ds\\
	&=T^+_{t_0}\phi(\gamma(-t_0))+\int_{-t_0}^{t}L(\gamma,\dot{\gamma})\ ds=T^-_{t_0+t}T^+_{t_0}\phi(x)=u(t,x).
\end{align*}
Since value functions $u$ and $\breve{u}$ are differentiable in the interior of the minimizer $\gamma$, we have that
\begin{align*}
	p(t)=D_xu(t,x)=D_x\breve{u}(t,x).
\end{align*}

(2) $\breve{u}\leqslant u$ is a direct consequence of Lemma \ref{lem:1} (1). In (1) we have already proved that
	\begin{align*}
	\mathcal{A}(\phi,t_0)\subset\{(t,x)\in[-t_0,0]\times M: u(t,x)=\breve{u}(t,x)\}.
    \end{align*}
On the other hand, for any $(t,x)\in[-t_0,0]\times M$ such that $\breve{u}(t,x)=u(t,x)$, there exists $\gamma_1:[t,0]\to M$, $\gamma_1(t)=x$ such that
\begin{align*}
	\breve{u}(t,x)=\phi(\gamma_1(0))-\int_{t}^{0}L(\gamma_1,\dot{\gamma}_1)\ ds,
\end{align*}
and there exists $\gamma_2:[-t_0,t]\to M$, $\gamma_2(t)=x$ such that
\begin{align*}
	u(t,x)=T^+_{t_0}\phi(\gamma_2(-t_0))+\int_{-t_0}^{t}L(\gamma_2,\dot{\gamma}_2)\ ds
\end{align*}
Let $\gamma=\overline{\gamma_2 \gamma_1}:[-t_0,0]\to M$, then $\gamma(t)=x$ and
\begin{align*}
	T^+_{t_0}\phi(\gamma(-t_0))=\phi(\gamma(0))-\int_{-t_0}^{0}L(\gamma,\dot{\gamma})\ ds,
\end{align*}
that is, $\gamma\in\mathscr{F}(\phi,t_0)$. This implies
\begin{align*}
	\{(t,x)\in[-t_0,0]\times M: u(t,x)=\breve{u}(t,x)\}\subset\mathcal{A}(\phi,t_0).
\end{align*}
Notice that $\breve{u}$ is locally semiconvex and $u$ is locally semiconcave on $(-t_0,0)\times M$. Proposition \ref{pro:cave geq vex} leads to
\begin{align*}
	Du=D\breve{u}\in\text{\rm Lip}_{loc}\,(\mathcal{A}(\phi,t_0)\setminus(\{-t_0,0\}\times M))
\end{align*}
If $t_0\in(0,\tau_1(\phi)]$, it follows from Proposition \ref{pro:T^-T^+1} (2) that $\mathcal{A}(\phi,t_0)=[-t_0,0]\times M$.

(3) For any $\gamma\in\mathscr{F}^{*}(\phi,t_0)$, let $p=L_v(\gamma,\dot{\gamma})$. Then, $p(0)\in D^{*}\phi(\gamma(0))=D^{*}T^-_{t_0}T^+_{t_0}\phi(\gamma(0))$. Thus, $\gamma$ is a minimizer for $T^-_{t_0}T^+_{t_0}\phi(\gamma(0))$, that is,
\begin{align*}
	\phi(\gamma(0))=T^-_{t_0}T^+_{t_0}\phi(\gamma(0))=T^+_{t_0}\phi(\gamma(-t_0))+\int_{-t_0}^{0}L(\gamma,\dot{\gamma})\ ds,
\end{align*}
which implies that $\gamma\in\mathscr{F}(\phi,t_0)$. Therefore, $\mathscr{F}^*(\phi,t_0)\subset\mathscr{F}(\phi,t_0)$ and $\mathcal{A}^*(\phi,t_0)\subset\mathcal{A}(\phi,t_0)$.
\end{proof}

Recall that, for any $\phi\in\text{\rm SCL}\,(M)$ and $t_0>0$, $T^-_{t_0}\circ T^+_{t_0}\phi=\phi$ if and only if $\phi=T^-_{t_0}\psi$ for some lower semicontinuous function $\psi$. We define
\begin{align*}
	\mathscr{F}(\phi,\psi,t_0)=&\,\left\{\gamma:[-t_0,0]\to M: \phi(\gamma(0))=\psi(\gamma(-t_0))+\int^0_{-t_0}L(\gamma,\dot{\gamma})\ ds\right\},\\
	\mathcal{A}(\phi,\psi,t_0)=&\,\{(t,\gamma(t)): \gamma\in \mathscr{F}(\phi,\psi,t_0), t\in[-t_0,0]\}.
\end{align*}
and
\begin{equation}\label{eq:u_phi}
	u_{\psi}(t,x):=T^-_{t_0+t}\psi(x),\qquad (t,x)\in[-t_0,0]\times M.
\end{equation}

\begin{Pro}\label{pro:control2}
Suppose $\phi\in\text{\rm SCL}\,(M)$, $t_0>0$ and $\phi=T^-_{t_0}\psi$. Let $u$ and $\breve{u}$ be defined in \eqref{eq:u_breve_u}, and let $u_{\psi}$ be defined in \eqref{eq:u_phi}. 
\begin{enumerate}[\rm (1)]
	\item We have the following relations:
	\begin{align*}
		\mathscr{F}^*(\phi,t_0)\subset\mathscr{F}(\phi,\psi,t_0)\subset\mathscr{F}(\phi,t_0),\qquad \mathcal{A}^*(\phi,t_0)\subset\mathcal{A}(\phi,\psi,t_0)\subset\mathcal{A}(\phi,t_0).
	\end{align*}
	\item $u_{\psi}\geqslant u\geqslant\breve{u}$ and
	\begin{align*}
		\mathcal{A}(\phi,\psi,t_0)=\{(t,x)\in[-t_0,0]\times M: u_{\psi}(t,x)=u(t,x)=\breve{u}(t,x)\}.
	\end{align*}
	Moreover, $Du_{\psi}=Du=D\breve{u}\in\text{\rm Lip}_{loc}\,(\mathcal{A}(\phi,\psi,t_0)\setminus(\{-t_0,0\}\times M))$.
\end{enumerate}
\end{Pro}

\begin{proof}
(1) The proof of $\mathscr{F}^*(\phi,t_0)\subset\mathscr{F}(\phi,\psi,t_0)$ is similar to the one of Proposition \ref{pro:control1} (3). For any $\gamma\in\mathscr{F}(\phi,\psi,t_0)$, we have that
\begin{align*}
	\psi(\gamma(-t_0))=\phi(\gamma(0))-\int_{-t_0}^{0}L(\gamma,\dot{\gamma})\ ds\leqslant T^+_{t_0}\phi(\gamma(-t_0)),
\end{align*}
while Lemma \ref{lem:1} (2) leads to $T^+_{t_0}\phi=T^+_{t_0}T^-_{t_0}\psi\leqslant\psi$. It follows that
\begin{align*}
	T^+_{t_0}\phi(\gamma(-t_0))=\psi(\gamma(-t_0))=\phi(\gamma(0))-\int_{-t_0}^{0}L(\gamma,\dot{\gamma})\ ds,
\end{align*}
which implies $\gamma\in\mathscr{F}(\phi,t_0)$. Therefore, $\mathscr{F}(\phi,\psi,t_0)\subset\mathscr{F}(\phi,t_0)$. Similarly, we have that
\begin{align*}
	\mathcal{A}^*(\phi,t_0)\subset\mathcal{A}(\phi,\psi,t_0)\subset\mathcal{A}(\phi,t_0).
\end{align*}

(2) Since $T^+_{t_0}\phi\leqslant\psi$, we have that
\begin{align*}
	u_{\psi}(t,x)=T^-_{t_0+t}\psi(x)\geqslant T^-_{t_0+t}T^+_{t_0}\phi(x)=u(t,x),\ \forall (t,x)\in[-t_0,0]\times M.
\end{align*}
For any $(t,x)\in\mathcal{A}(\phi,\psi,t_0)$, there exists $\gamma\in\mathscr{F}(\phi,\psi,t_0)\subset\mathscr{F}(\phi,t_0)$ such that $\gamma(t)=x$. Thus, we have that
\begin{align*}
	u_{\psi}(t,x)=\psi(\gamma(-t_0))+\int_{-t_0}^{t}L(\gamma,\dot{\gamma})\ ds=\phi(\gamma(0))-\int_{t}^{0}L(\gamma,\dot{\gamma})\ ds=\breve{u}(t,x).
\end{align*}
This implies that
\begin{align*}
	\mathcal{A}(\phi,\psi,t_0)\subset\{(t,x)\in[-t_0,0]\times M: u_{\psi}(t,x)=u(t,x)=\breve{u}(t,x)\}.
\end{align*}
The remaining part of the proof is similar to Proposition \ref{pro:control1} (2).
\end{proof}

\begin{Cor}\label{cor:cave t source}
Under the same assumptions of Proposition \ref{pro:control2}, the following statements are equivalent.
\begin{enumerate}[\rm (1)]
	\item $\phi=T^-_{t_0}\psi$.
	\item $u_{\psi}:[-t_0,0]\times M\to\R$ is a viscosity solution of the Hamilton-Jacobi equation
	\begin{align*}
		\begin{cases}
			D_tu+H(x,D_xu)=0,&(t,x)\in [-t_0,0)\times M,\\
			u(0,x)=\phi(x),& x\in M.
		\end{cases}
	\end{align*}
	\item $\psi(x)=T^+_{t_0}\phi(x)$ for all $x\in\pi_x(\mathcal{A}^*(\phi,t_0)\cap(\{-t_0\}\times M))$, and $T^+_{t_0}\phi(x)\leqslant\psi(x)$ otherwise.
\end{enumerate}	
\end{Cor}

\begin{proof}
(1) $\Leftrightarrow$ (2) is trivial. (1) $\Rightarrow$ (3) follows directly from Proposition \ref{pro:control2}. We only need to prove that (3) $\Rightarrow$ (1). $\psi\geqslant T^+_{t_0}\phi$ implies $T^-_{t_0}\psi\geqslant T^-_{t_0}T^+_{t_0}\phi=\phi$. On the other hand, for any $x\in M$, choose $p\in D^{*}\phi(x)=D^{*}T^-_{t_0}T^+_{t_0}\phi(x)$ and let $\gamma(s)=\pi_x\Phi_H^s(x,p)$, $s\in[-t_0,0]$. It follows that
\begin{align*}
	\phi(x)=T^-_{t_0}T^+_{t_0}\phi(x)=T^+_{t_0}\phi(\gamma(-t_0))+\int_{-t_0}^{0}L(\gamma,\dot{\gamma})\ ds.
\end{align*}
Notice that $\gamma(-t_0)\in\pi_x(\mathcal{A}^*(\phi,t_0)\cap(\{-t_0\}\times M))$, which implies $T^+_{t_0}\phi(\gamma(-t_0))=\psi(\gamma(-t_0))$. Thus, we have
\begin{align*}
	\phi(x)=\psi(\gamma(-t_0))+\int_{-t_0}^{0}L(\gamma,\dot{\gamma})\ ds\geqslant T^-_{t_0}\psi(x).
\end{align*}
In conclusion, there holds $T^-_{t_0}\psi(x)=\phi(x)$, $\forall x\in M$.
\end{proof}

\appendix

\section{some facts from weak KAM theory}

\begin{Pro}\label{pro:regularity}
Suppose $L$ is a Tonelli Lagrangian. Then for any $\lambda>0$
\begin{enumerate}[\rm (1)]
	\item there exists a constant $C_\lambda>0$ such that for any $x\in\R^n$, $t\in(0,2/3)$ the function $y\mapsto A_t(x,y)$ defined on $B(x,\lambda t)$ is semiconcave with constant $\frac{C_{\lambda}}{t}$ uniformly with respect to $x$; 
	\item there exist $C'_\lambda>0$ and $t_{\lambda}>0$ such that the function $y\mapsto A_t(x,y)$ is convex with constant $\frac{C'_{\lambda}}{t}$ on $B(x,\lambda t)$ with $0<t\leqslant t_{\lambda}$. The constants $C'_{\lambda}$ and $t_{\lambda}$ are independent of $x$;
	\item there exists $t'_{\lambda}>0$ such that the function $y\mapsto A_t(x,y)$ is of class $C^2$ on $B(x,\lambda t)$ with $0<t\leqslant t'_{\lambda}$. Moreover, 
	\begin{align*}
		D_yA_t(x,y)=&L_v(\xi(t),\dot{\xi}(t)),\\
		D_xA_t(x,y)=&-L_v(\xi(0),\dot{\xi}(0)),\\
		D_tA_t(x,y)=&-E_{t,x,y},\label{eq:diff_A_t_t}
	\end{align*}
	where $\xi\in\Gamma^t_{x,y}$ is the unique minimizer for $A_t(x,y)$ and
	\begin{align*}
		E_{t,x,y}:=E(\xi(s),\dot{\xi}(s))\qquad\forall \,s\in[0,t].
	\end{align*}
	We remark that $E(x,v):=L_v(x,v)\cdot v-L(x,v)$ is the energy function in the Lagrangian formalism, and
	\begin{align*}
		E(\xi(s),\dot{\xi}(s))=H(\xi(s),p(s)),\quad s\in[0,t],
	\end{align*}
	for the dual arc $p(s)=L_v(\xi(s),\dot{\xi}(s))$;
\end{enumerate} 
\end{Pro}

For the proof of the aforementioned Proposition, the readers can refer to \cite{Cannarsa_Cheng3} and \cite{Cannarsa_Cheng2021a}.

\begin{Lem}[\cite{Carneiro1995}]\label{lem:app_appr_2}
For any $\varepsilon>0$ there exists $r_2(\varepsilon)>0$ such that if $t>r_2(\varepsilon)$ and $\gamma\in\Gamma^t_{x,y}$ is a minimal curve for $A_t(x,y)$, and $p=L_v(\gamma,\dot{\gamma})$ is the dual arc, then (recall that $H(\gamma(s),p(s))=const$)
\begin{align*}
	|H(\gamma(s),p(s))-c[H]|<\varepsilon,\quad\forall s\in[0,t].
\end{align*}
\end{Lem}

%

\bibliographystyle{plain}
\bibliography{mybib}

\begin{thebibliography}{10}

\bibitem{Ambrosio_Brue_Semola_book2021}
Luigi Ambrosio, Elia Bru\'{e}, and Daniele Semola.
\newblock {\em Lectures on optimal transport}, volume 130 of {\em Unitext}.
\newblock Springer, 2021.

\bibitem{Ancona_Cannarsa_Nguyen2016_2}
Fabio Ancona, Piermarco Cannarsa, and Khai~T. Nguyen.
\newblock Compactness estimates for {H}amilton-{J}acobi equations depending on
  space.
\newblock {\em Bull. Inst. Math. Acad. Sin. (N.S.)}, 11(1):63--113, 2016.

\bibitem{Ancona_Cannarsa_Nguyen2016_1}
Fabio Ancona, Piermarco Cannarsa, and Khai~T. Nguyen.
\newblock Quantitative compactness estimates for {H}amilton-{J}acobi equations.
\newblock {\em Arch. Ration. Mech. Anal.}, 219(2):793--828, 2016.

\bibitem{Arnaud2011}
M.-C. Arnaud.
\newblock Pseudographs and the {L}ax-{O}leinik semi-group: a geometric and
  dynamical interpretation.
\newblock {\em Nonlinearity}, 24(1):71--78, 2011.

\bibitem{Bardi_Capuzzo-Dolcetta1997}
Martino Bardi and Italo Capuzzo-Dolcetta.
\newblock {\em Optimal control and viscosity solutions of
  {H}amilton-{J}acobi-{B}ellman equations}.
\newblock Systems \& Control: Foundations \& Applications. Birkh{\"a}user
  Boston, Inc., Boston, MA, 1997.
\newblock With appendices by Maurizio Falcone and Pierpaolo Soravia.

\bibitem{BCJS1999}
E.~N. Barron, P.~Cannarsa, R.~Jensen, and C.~Sinestrari.
\newblock Regularity of {H}amilton-{J}acobi equations when forward is backward.
\newblock {\em Indiana Univ. Math. J.}, 48(2):385--409, 1999.

\bibitem{Bernard2007}
Patrick Bernard.
\newblock Existence of {$C^{1,1}$} critical sub-solutions of the
  {H}amilton-{J}acobi equation on compact manifolds.
\newblock {\em Ann. Sci. {\'E}cole Norm. Sup. (4)}, 40(3):445--452, 2007.

\bibitem{Bernard2010}
Patrick Bernard.
\newblock Lasry-{L}ions regularization and a lemma of {I}lmanen.
\newblock {\em Rend. Semin. Mat. Univ. Padova}, 124:221--229, 2010.

\bibitem{Bernard2012}
Patrick Bernard.
\newblock The {L}ax-{O}leinik semi-group: a {H}amiltonian point of view.
\newblock {\em Proc. Roy. Soc. Edinburgh Sect. A}, 142(6):1131--1177, 2012.

\bibitem{Bernard_Buffoni2007a}
Patrick Bernard and Boris Buffoni.
\newblock Optimal mass transportation and {M}ather theory.
\newblock {\em J. Eur. Math. Soc. (JEMS)}, 9(1):85--121, 2007.

\bibitem{Bernard_Buffoni2007b}
Patrick Bernard and Boris Buffoni.
\newblock Weak {KAM} pairs and {M}onge-{K}antorovich duality.
\newblock In {\em Asymptotic analysis and singularities---elliptic and
  parabolic {PDE}s and related problems}, volume~47 of {\em Adv. Stud. Pure
  Math.}, pages 397--420. Math. Soc. Japan, Tokyo, 2007.

\bibitem{Cannarsa_Cheng3}
Piermarco Cannarsa and Wei Cheng.
\newblock Generalized characteristics and {L}ax-{O}leinik operators: global
  theory.
\newblock {\em Calc. Var. Partial Differential Equations}, 56(5):Art. 125, 31,
  2017.

\bibitem{Cannarsa_Cheng2021a}
Piermarco Cannarsa and Wei Cheng.
\newblock Singularities of {S}olutions of {H}amilton--{J}acobi {E}quations.
\newblock {\em Milan J. Math.}, 89(1):187--215, 2021.

\bibitem{Cannarsa_Cheng_Fathi2017}
Piermarco Cannarsa, Wei Cheng, and Albert Fathi.
\newblock On the topology of the set of singularities of a solution to the
  {H}amilton-{J}acobi equation.
\newblock {\em C. R. Math. Acad. Sci. Paris}, 355(2):176--180, 2017.

\bibitem{Cannarsa_Cheng_Fathi2021}
Piermarco Cannarsa, Wei Cheng, and Albert Fathi.
\newblock Singularities of solutions of time dependent {H}amilton-{J}acobi
  equations. {A}pplications to {R}iemannian geometry.
\newblock {\em Publ. Math. Inst. Hautes \'{E}tudes Sci.}, 133(1):327--366,
  2021.

\bibitem{Cannarsa_Sinestrari_book}
Piermarco Cannarsa and Carlo Sinestrari.
\newblock {\em Semiconcave functions, {H}amilton-{J}acobi equations, and
  optimal control}, volume~58 of {\em Progress in Nonlinear Differential
  Equations and their Applications}.
\newblock Birkh{\"a}user Boston, Inc., Boston, MA, 2004.

\bibitem{Carneiro1995}
M.~J.~Dias Carneiro.
\newblock On minimizing measures of the action of autonomous {L}agrangians.
\newblock {\em Nonlinearity}, 8(6):1077--1085, 1995.

\bibitem{Esteve-Yague_Zuazua2023}
Carlos Esteve-Yag\"{u}e and Enrique Zuazua.
\newblock Reachable set for {H}amilton-{J}acobi equations with non-smooth
  {H}amiltonian and scalar conservation laws.
\newblock {\em Nonlinear Anal.}, 227:Paper No. 113167, 18, 2023.

\bibitem{Fathi_book}
Albert Fathi.
\newblock {W}eak {KAM} theorem in {L}agrangian dynamics.
\newblock Cambridge University Press, Cambridge (to appear).

\bibitem{Fathi1997_2}
Albert Fathi.
\newblock Solutions {KAM} faibles conjugu{\'e}es et barri{\`e}res de {P}eierls.
\newblock {\em C. R. Acad. Sci. Paris S{\'e}r. I Math.}, 325(6):649--652, 1997.

\bibitem{Fathi1998_1}
Albert Fathi.
\newblock Sur la convergence du semi-groupe de {L}ax-{O}leinik.
\newblock {\em C. R. Acad. Sci. Paris S{\'e}r. I Math.}, 327(3):267--270, 1998.

\bibitem{Fathi_Figalli2010}
Albert Fathi and Alessio Figalli.
\newblock Optimal transportation on non-compact manifolds.
\newblock {\em Israel J. Math.}, 175:1--59, 2010.

\bibitem{Lions_book}
Pierre-Louis Lions.
\newblock {\em Generalized solutions of {H}amilton-{J}acobi equations},
  volume~69 of {\em Research Notes in Mathematics}.
\newblock Pitman (Advanced Publishing Program), Boston, Mass.-London, 1982.

\bibitem{Mather1993}
John~N. Mather.
\newblock Variational construction of connecting orbits.
\newblock {\em Ann. Inst. Fourier (Grenoble)}, 43(5):1349--1386, 1993.

\bibitem{Udricste_Udricste1994}
Constantin Udri\c{s}te and Constantin Udri\c{s}te.
\newblock {\em Convex functions and optimization methods on {R}iemannian
  manifolds}, volume 297 of {\em Mathematics and its Applications}.
\newblock Kluwer Academic Publishers Group, Dordrecht, 1994.

\bibitem{Villani_book2009}
C\'{e}dric Villani.
\newblock {\em Optimal transport: old and new}, volume 338 of {\em Grundlehren
  der Mathematischen Wissenschaften}.
\newblock Springer-Verlag, Berlin, 2009.

\end{thebibliography}

\end{document}